\documentclass{article}
\usepackage{amsthm, amssymb, amsmath}
\usepackage{verbatim}
\newtheorem{theorem}{Theorem}[section]
\newtheorem{definition}[theorem]{Definition}
\newtheorem{lemma}[theorem]{Lemma}

\theoremstyle{plain}
\newtheorem{remark}[theorem]{Remark}

\def \R {\mathbb R}
\def \B {\mathbb B}
\def \sp {\mathbb S}

\def \vol {\operatorname{vol}}

\def \M {M_{+,p}}
\def \Me {M_{\varepsilon,p}}
\def \Md {\M^\circ}
\def \Med {\Me^\circ}
\def \vol {\operatorname{vol}}
\def \co {\operatorname{co}}

\def \supp {\operatorname{supp}}
\def \domlambda {(\frac n{n+p},1) \cup (1, +\infty]}
\def \s{s_{f,p}}
\def \c{c_{f,p}}

\def \gl {\operatorname{GL}}

\def \l {\lambda}

\def \constanta {a_{n,p,\l}}
\def \constantb {b_{n,p,\l}}
\def \constantd {d_{n,p,\l}}

\author{J. Haddad}
\author{C. H. Jim\'enez}
\author{J. Haddad, C. H. Jim\'enez, M. Montenegro}
\begin{document}
\title{Asymmetric Blaschke-Santal\'o functional inequalities}
\date{}
\maketitle

\abstract{ In this work we establish functional asymmetric versions of the celebrated Blaschke-Santal\'o inequality. As consequences of these inequalities we recover their geometric counterparts with equality cases, as well as, another inequality with strong probabilistic flavour that was firstly obtained by Lutwak, Yang and Zhang. We present a brief study on an $L_p$ functional analogue to the center of mass that is necessary for our arguments and that might be of independent interest.}

\section{Introduction}

Functional inequalities with geometric counterpart have attracted great interest in recent years. The  Brunn-Minkowski/Prekopa-Leindler inequality \cite{Prk} or Petty-Projection/Sobolev inequality \cite{Zh} are just two of the main examples in this direction. While the first connection above helped trigger a fruitful development of functional analogues  of several geometric parameters  into the class of $\log$-concave functions currently undergoing, the second confirmed the strong links between Sobolev type inequalities and isoperimetric inequalites, generating at the same time an increasing interest for finding stronger, affine invariant or asymmetric versions of both Sobolev type inequalities and isoperimetric inequalities.

Within the class of isoperimetric inequalities, which often compare two parameters associated to convex, star or other more general bodies, we will deal in this work with the subclass of affine isoperimetric inequalities. These inequalities remain invariant under non-degenerate linear transformations (or other subgroups) and its interest lies in the fact that they often are stronger and  imply their Euclidean counterparts.

More specifically, our work concerns the Blaschke-Santal\'o inequality for convex sets and some functional versions of it. The Blaschke-Santal\'o inequality is undeniably one of the most important affine invariant inequalities in Convex Geometry (see \cite{Sch} for an overview and references therein).
Let $K\subset\R^n$ be an origin-symmetric convex body (compact and convex with nonempty interior such that $K=-K$) and $K^\circ$ the polar body defined by
\[K^\circ = \{y \in \R^n: \langle y, z \rangle \leq 1\ \ \forall z \in K \}\, .\]
The Blaschke-Santal\'o inequality states that

\begin{equation}
	\label{ineq_BSsim}
	\vol(K) \vol(K^\circ) \leq \omega_n^2
\end{equation}
and equality holds if, and only if, $K$ is a centered ellipsoid, where $\omega_n$ is the volume of the Euclidean unit ball $\B^n_2 \subset \R^n$ for $n \geq 2$.

If $K$ is not origin-symmetric, then there is a unique point $s = s(K) \in K$ for which the volume of $(K-s)^\circ$ is minimized (see \cite{Sch} for references).
This point $s$ is called the Santal\'o point of $K$ and we have
\begin{equation}
	\label{ineq_BSasim}
	\vol(K) \vol((K-s)^\circ) \leq \omega_n^2
\end{equation}
and equality holds if, and only if, $K$ is an ellipsoid.

It is interesting to notice that the product volume in the above Blaschke-Santal\'o inequality remains bounded by $\omega_n^2$ when the Santal\'o point $s$ of $K$ is replaced by the center of mass $\bold{c} = \bold{c}(K)$ of $K$. In fact, we have

\begin{equation} \label{ineq BScentr}
\vol(K) \vol((K-\bold{c})^\circ) \leq \omega_n^2
\end{equation}
and equality holds if, and only if, $K$ is an ellipsoid. This claim can be proved from the following result which provides a characterization of Santal\'o points:

\begin{lemma}[see {\cite[10.22]{Sch}}]
	\label{lemma_center}
Let $K \subset \R^n$ be a convex body. Then $0$ is the Santal\'o point of $K$ if, and only if,

\[
\bold{c}(K^\circ) = \int_{K^\circ} z dz = 0\, .
\]
In other words, $0$ is the Santal\'o point of $K$ if, and only if, $K^\circ$ has its center of mass at the origin.
\end{lemma}
Applying Inequality \eqref{ineq_BSasim} to the set $(K - \bold{c}(K))^\circ$ instead of $K$ and using the previous lemma we obtain that both $s(K)$ and $\bold{c}(K)$ belong to the set

\begin{equation}\label{Santaloreg}
\left\{x\in K\, :\, \frac{\vol(K)\vol((K-x)^\circ)}{\omega_n^2}\leq 1\right\},
\end{equation}
and these two points coincide when $K$ is an ellipsoid. This latter remark will shed some light on the appearance of an $Lp$-funcional analogue of the center of mass that we study in section 3. Let us note that the set in (\ref{Santaloreg}) was considered in \cite{M-W} where the authors studied its size and other properties.
%the Santal\'o point $s$ in \eqref{ineq_BSasim} can be replaced by %$\bold{c}(K)$.

Blaschke-Santal\'o inequalities have great influence in different areas. In particular, the Blaschke-Santal\'o inequality \eqref{ineq_BSasim} is equivalent to the well-known affine isoperimetric inequality (see formula (10.19) and the discussion on page 548 of \cite{Sch}). Inequalities \eqref{ineq_BSsim} and \eqref{ineq_BSasim} were established by Blaschke \cite{Blas} for $n=2$ and $n=3$ and later by Santal\'o \cite{Santa} for arbitrary dimensions and, finally, the cases of equality were proved by Petty \cite{Pet}. We also refer to Hug \cite{Hug}, Meyer and Pajor \cite{MP} and Meyer and Reisner \cite{MR} for simpler proofs. More recently, Bianchi and Kelly \cite{BK} provided a proof of \eqref{ineq_BSsim} by using a Fourier analytic approach. The Orlicz case for star bodies was treated by Lutwak, Yang and Zhang \cite{LYZorl} and Zhu \cite{Zh1}. Functional versions were established by Fradelizi and Meyer \cite{FM} in relation to the Pr\'ekopa-Leindler inequality, by Artstein-Avidan, Klartag and Milman \cite{A-K-M} and later by Lin and Leng \cite{LiLe} for log-concave functions.

Lutwak and Zhang \cite{LZ} presented a different proof of the Blaschke-Santal\'o inequality \eqref{ineq_BSsim} as a limit case as $p \to + \infty$ of a family of $L_p$ affine inequalities for centroid bodies. Let $K$ be a star body (compact and star-shaped with nonempty interior) with respect to the origin, $p\geq 1$ and $\Gamma_p K$ be the convex body defined by means of the support function

\[
h(\Gamma_p K, y) =\left( \frac 1 {c_{n,p} \vol(K)} \int_K |\langle y, z \rangle|^p dz \right)^{\frac{1}{p}},
\]
where the normalization constant $c_{n,p}$ is chosen so that $\Gamma_p \B^n_2 = \B^n_2$. Then,
\begin{equation}
	\label{ineq_BSlp}
	\vol(K) \vol(\Gamma_p^\circ K) \leq \omega_n^2\, .
\end{equation}
Inequality \eqref{ineq_BSlp} is an extension of \eqref{ineq_BSsim} in the sense that $h(\Gamma_p K, \cdot)$ converges to the support function of the convex hull of $K\cup -K$ when $p \to + \infty$, so that the Blaschke-Santal\'o inequality can be recovered when $K$ is convex and symmetric. Moreover, they also gave an equivalent functional version of this inequality.
Let $\sp^{n-1}$ denote the Euclidean unit sphere for $n \geq 2$. Then given two positive continuous functions $f,g:\sp^{n-1} \to \R$, we have
\begin{equation}
	\label{ineq_BSlp_func_fg}
	c_{n-2,p} \|f\|_{\frac n{n+p}} \|g\|_{\frac n{n+p}} \leq \int_{\sp^{n-1}} \int_{\sp^{n-1}} |\langle x, y \rangle|^p f(x) g(y) dx dy\, .
\end{equation}

In \cite{LYZ}, Lutwak, Yang and Zhang proved the $L_p$ Busemann-Petty centroid inequality
\begin{equation}
	\label{ineq_BP}
	\vol(\Gamma_p K) \geq \vol(K)\, ,
\end{equation}
which will play a central role in our paper.
It is worth noticing that \eqref{ineq_BP} together with \eqref{ineq_BSsim} implies \eqref{ineq_BSlp}.

More recently, Haberl and Schuster \cite{HS} proved some inequalities extending \eqref{ineq_BSlp} and \eqref{ineq_BP} to the asymmetric case.
Defining the asymmetric $L_p$ moment body $\Me K$ by the support function
\[
h(\Me K, y) =\left( \int_K \langle y, z \rangle_\varepsilon^p dz \right)^{\frac{1}{p}},
\]
where $\varepsilon \in [0,1]$ and $\langle y, z \rangle_\varepsilon^p = (1-\varepsilon) \max\{ \langle y, z \rangle, 0\}^p + \varepsilon \max\{ - \langle y, z \rangle, 0\}^p$, they extended inequalities \eqref{ineq_BSlp} and \eqref{ineq_BP} for the operator $\Me$, namely
\begin{equation}
	 \label{ineq_BSlp_asim}
	 \vol(K)^{\frac{n}{p}+1} \vol\left(((\Me K)-s)^\circ\right) \leq R_{n,p}
\end{equation}
and

\begin{equation}
	\label{ineq_BP_asim}
	 \vol(K)^{-\frac{n}{p}-1}\vol(\Me K) \geq r_{n,p}\, ,
\end{equation}
where $s \in \R^n$ is the Santal\'o point of the set $\Me K$, which is not necessarily origin-symmetric.
Here the constants $R_{n,p}$ and $r_{n,p}$ are defined as the left-hand side of the respective inequality for $K = \B^n_2$, precisely
\[
R_{n,p} = {\pi^{\frac{n}{2 p}+n} {\Gamma \left(\frac{n+2}{2}\right)^{-2}} \left(\frac{2 \Gamma \left(\frac{1}{2} (n+p+2)\right)}{\Gamma \left(\frac{n+2}{2}\right) \Gamma \left(\frac{p+1}{2}\right)}\right)^{\frac{n}{p}}}
\]
and

\[
r_{n,p} = {\pi^{-\frac{n}{2 p}} \left(\frac{2 \Gamma \left(\frac{1}{2} (n+p+2)\right)}{\Gamma \left(\frac{n+2}{2}\right) \Gamma \left(\frac{p+1}{2}\right)}\right)^{-\frac{n}{p}}}.
\]

Inequalities \eqref{ineq_BSsim}, \eqref{ineq_BSasim}, \eqref{ineq_BSlp}, \eqref{ineq_BP}, \eqref{ineq_BSlp_asim} and \eqref{ineq_BP_asim} are sharp and equality holds in all of them if, and only if, $K$ is an origin-symmetric ellipsoid.

The bodies $\Me K$ were introduced by Ludwig \cite{Lud}, and afterwards several affine functional inequalities related to the moment and centroid bodies have been extended to the asymmetric case, see for example \cite{Wa} and \cite{Ng}.

Among the several strong (and perhaps surprising) connections between Probability and Convex Geometry (see \cite{Madi} and references therein), in \cite{RENYI} Lutwak, Yang and Zhang obtained an inequality that provides a sharp lower bound for the moments of the inner product of random variables in terms of their $\l$-R\'enyi entropy. Their inequality, that relies on the theory of dual mixed volumes, can be regarded as a generalization of \eqref{ineq_BSlp_func_fg} for functions supported in $\R^n$. Namely, let $\lambda \in \domlambda$ and $\l' = \frac{\l}{\l - 1}$ be the dual of $\l$. Then, for any nonzero nonnegative continuous functions $f,g:\R^n \to \R$ with compact support,

\begin{equation}
	\label{ineq_Renyi}
	\|f\|_{1}^{n+\lambda'p} \|g\|_{1}^{n+\lambda'p} \leq  d_{n,p,\l} \|f\|_\l^{\l' p} \|g\|_\l^{\l' p} \left( \int_{\R^n} \int_{\R^n} f(x) g(y) |\langle x, y \rangle|^p dx dy \right)^{n}.\\
\end{equation}

We are interested in the following functional version of the $L_p$ Blaschke-Santal\'o inequality \eqref{ineq_BSlp}. Let $\l \in \domlambda$. Then, for any nonzero nonnegative continuous function $f:\R^n \to \R$ with compact support,

\begin{equation}
	\label{ineq_BSlp_func_sim}
	\|f\|_{1}^{n+\lambda' p} \leq b_{n,p,\l} \|f\|_{\lambda}^{\lambda'p} \left( \int_{\sp^{n-1}} \left(\int_{\R^n} f(x) |\langle x, \xi \rangle|^p dy \right)^{-\frac{n}{p}} d\xi \right)^{-p}.
\end{equation}
Noticing that the formula for $h(M_pK, y)$ is a convex function regardless of the set $K$, we might interpret the double integral of \eqref{ineq_BSlp_func_sim} as the volume of the polar $L_p$ centroid body of $f$ (see Definition \ref{def_centroidf} below).
Plugging $f = \chi_K$ in \eqref{ineq_BSlp_func_sim} for $\lambda = + \infty$, one obtains \eqref{ineq_BSlp} for a general measurable set $K$.

Inequality \eqref{ineq_BSlp_func_sim} can be easily deduced in several ways: by a radial symmetrization argument (Steiner symmetrization with respect to the radial direction in volume-preserving polar coordinates), Inequality \eqref{ineq_BSlp} is immediately extended to any bounded measurable set $K \subset \R^n$.
Then, applying this extension to the level sets of $f$, one derives the result. However, one cannot apply this argument in the asymmetric case because the Santal\'o point varies with the level set. A shorter proof follows from \eqref{ineq_Renyi} by taking a suitable $g$ depending on $f$.
The fact that there is a connection between \eqref{ineq_Renyi} and \eqref{ineq_BSlp_func_sim} becomes clearer taking into account that the proof of Inequality \eqref{ineq_Renyi} in \cite{RENYI} uses a radial body $S_p f$ that for $f=\chi_K$ is a radial symmetrization of $K$.

In recent works \cite{HJM2, HJM1, HJM3}, the authors showed that the $L_p$ Busemann-Petty centroid inequality \eqref{ineq_BP} may be used to transform sharp Sobolev type inequalities with general norms into affine invariant inequalities. This is the case, for example, of the $L^p$ Sobolev inequality with a general norm that was proved by Cordero-Erausquin, Nazaret and Villani in \cite{CNV} using a mass-transportation approach. Let $p \in [1,n)$ and $\|.\|$ be a norm in $\R^n$. Then, for any smooth function $f:\R^n \to \R$ with compact support,
\begin{equation}
	\label{ineq_sobolevGNV}
	\|f\|_{\frac{np}{n-p}} \leq S_{n,p,\|.\|} \left( \int_{\R^n} \|\nabla f(x)\|_*^p dx \right)^{\frac{1}{p}}.
\end{equation}
Here $\|.\|_*$ is the dual norm of $\|.\|$ (see Section 2) and $S_{n,p,\|.\|}$ denotes the best Sobolev constant for \eqref{ineq_sobolevGNV}. Following \cite{HJM1}, if we choose a specific norm $\|.\|_f$ for each $f$ and apply the $L_p$ Busemann-Petty centroid inequality, we end up with the important sharp affine Sobolev inequality of Lutwak, Yang and Zhang \cite{LYZp}
\begin{equation}
	\label{ineq_sobolevLYZ}
	\|f\|_{\frac{np}{n-p}} \leq S_{n,p} \left( \int_{\sp^{n-1}} \left( \int_{\R^n} |\nabla_\xi f(x)|^p dx \right)^{-\frac{n}{p}}\right)^{-\frac{1}{n}},
\end{equation}
where $S_{n,p}$ and $\nabla_\xi f(x)$ denote, respectively, the best Sobolev constant for the Euclidean norm and the directional derivative of $f$ at the point $x \in \R^n$ in the direction of $\xi \in \sp^{n-1}$.

In the present work, we analyze a similar technique regarding a moment inequality with a general norm proved in \cite{RENYI}, and show that it may be transformed into a functional version of the $L_p$ Blaschke-Santal\'o inequality. We begin by recalling the following result due to Lutwak, Yang and Zhang:

\begin{lemma}[{\cite[Lemma 4.1]{RENYI}}]
Let $\l \in \domlambda$ and $K$ be a convex body with the origin in its interior. Denote by $g$ the gauge function of $K$ defined by $g(K,x) = \inf\{\lambda > 0:\, \lambda x \in K\}$. Then, for any nonzero nonnegative continuous function $f:\R^n \to \R$ with compact support,

\begin{equation}
	\label{ineq_moment}
	\|f\|_1^{n+p \l '} \leq \constanta \left( \int_{\R^n} f(x) g(K,x)^p dx\right)^{n} \|f\|_\l^{p \l '} \vol(K)^{p}.
\end{equation}
The best constant $\constanta$ is given by

\[
	\constanta = \left\{
	\begin{array}{cc}
		\left(\frac np \left(\frac{n}{\l' p}+1\right)^{\l'-1} \left(\frac{\lambda  p}{(1-\lambda ) n}-1\right)^{\frac{n}{p}} B\left(\frac{n}{p},-\l'-\frac{n}{p}+1\right)\right)^{p} & \l < 1\\
		\left(\frac np \left(\frac{n}{\l' p}+1\right)^{\l'-1} \left(\frac{\l' p}{n}+1\right)^{\frac{n}{p}} B\left(\frac{n}{p},\l'\right)\right)^{p} & \l > 1\\
	\left(\frac{n}{n+p}\right)^{-n} & \l = + \infty\\
	\end{array}
	\right.
	,
\]
where $B(\cdot, \cdot)$ denotes the binomial function

\[
B(m,l) = \frac{\Gamma(m)\Gamma(l)}{\Gamma(m+l)}
\]
for $m,l > 0$, being $\Gamma(\cdot)$ the usual Gamma function. Moreover, Inequality \eqref{ineq_moment} is sharp and equality holds if, and only if, $f(x) = a p_\l(b g(K,x))$ for constants $a, b \in \R$, where
	\begin{align*}
		p_\l(s) = \left\{
			\begin{array}{cc}
				(1+|s|^p)^{1/(\l-1)}, & \l < 1\\
				(1-|s|^p)_+^{1/(\l-1)}, & \l > 1\\
				\chi_{[-1,1]}(s), & \l = + \infty\\
			\end{array}
			\right.\, .
	\end{align*}
Here $t_+ = \max\{t,0\}$ and $\chi_{[-1,1]}$ denotes the characteristic function of the set $[-1,1]$.
\end{lemma}

Our main results consist in asymmetric versions of inequalities \eqref{ineq_BSlp_func_sim} and \eqref{ineq_Renyi}, both inequalities obtained directly from \eqref{ineq_moment} and \eqref{ineq_BSlp_asim}.
The major difficulty is to establish a suitable definition of a Santal\'o point for $f$ (see Theorem \ref{thm_existeSantaloLp}).

We shall prove the following two theorems:

\begin{theorem}
	\label{thm_BSlp_func_asim}
	Let $\l \in \domlambda$. We consider two separate cases:

	\begin{enumerate}
		\item Symmetric case: For any nonzero nonnegative continuous function $f:\R^n \to \R$ with compact support,

	\begin{equation} \label{ineq_BSlp_func_asim_12}
		\|f\|_{1}^{n+\l' p}	\leq 2^{-n} \constantb \|f\|_{\l}^{\l'p} \left\{ \int_{\sp^{n-1}} \left(\int_{\R^n} f(x) |\langle x, \xi \rangle|^p dx \right)^{-\frac np} d\xi \right\}^{-p} .
	\end{equation}

	The best constant $\constantb$ is given by
	\[\constantb = \left(\frac{n}{n+p}\right)^{n} \constanta R_{n,p}^{p}n^p.\]

	Moreover, Inequality \eqref{ineq_BSlp_func_asim_12} is sharp and equality holds if, and only if, $f(x) = a p_\l(|B x|)$, where $a \in \R$ and $B \in \gl_n$. Here $\gl_n$ denotes the set of invertible $n \times n$-matrices.

	\item Asymmetric case: For each nonzero nonnegative continuous function $f:\R^n \to \R$ with compact support, there exists a point $\c \in \R^n$ such that for any $\varepsilon \in [0, \frac 12)$,
	\begin{equation} \label{ineq_BSlp_func_asim}
		\|f\|_{1}^{n+\l' p} \leq \constantb \|f\|_{\l}^{\l'p} \left\{ \int_{\sp^{n-1}} \left(\int_{\R^n} f(x - \c) \langle x, \xi \rangle_\varepsilon^p dx \right)^{-\frac np} d\xi \right\}^{-p} .
	\end{equation}

	Moreover, Inequality \eqref{ineq_BSlp_func_asim} is sharp and equality holds if, and only if, $f(x) = a p_\l(|B(x-\c)|)$, where $a \in \R$ and $B \in \gl_n$.

	\end{enumerate}

\end{theorem}

It is easy to see, using the ideas in the proof of Theorem \ref{thm_existeSantaloLp}, that such a point $\c$ in \eqref{ineq_BSlp_func_asim} must belong to $\co( \supp f)$. While inequalities \eqref{ineq_BSlp_func_asim_12} and \eqref{ineq_BSlp_func_asim} are valid for continuous functions $f$ with compact support, they can be extended to $L^1\left(\R^n, (1 + |x|^p) dx\right) \cap L^\l(\R^n)$.
Taking $\varepsilon = 0$ and $f(x) = p_\l(g(K,x))$ for $K$ a star-body, one recovers
\begin{equation}
	\label{ineq_BSlp_asim_new}
	\vol(K)^{\frac{n}{p}+1} \min_{s\in\R^n} \left\{ \vol(M_{0,p}^\circ (K-s)) \right\} \leq R_{n,p}.
\end{equation}
The same conclusion holds for $\l = + \infty$ and $f = \chi_K$ with $K \subseteq \R^n$ any bounded measurable set.
The reader should compare with Inequality \eqref{ineq_BSlp_asim}, see also Remark \ref{rem_uniqueness}.

It should be emphasized that the content of the second part of Theorem \ref{thm_BSlp_func_asim} reduces to the case $\varepsilon=0$ for which the right-hand side is minimized.

%\todo{Decir algo aqui? o despues del teorema }
\begin{theorem}
	Let $\l \in \domlambda$ and $\varepsilon \in [0, \frac 12]$. Then, for each nonzero nonnegative continuous function $f:\R^n \to \R$ with compact support, there exists a point $\c \in \R^n$ such that
	\label{thm_Renyi_asim}
	\begin{equation}
		\label{ineq_Renyi_asim}
		\|f\|_{1}^{n+\l'p} \|g\|_{1}^{n+\l'p} \leq  \constantd \|f\|_{\l}^{\l'p } \|g\|_{\l}^{\l'p} \left( \int_{\R^n} \int_{\R^n} f(x - \c) g(y) \langle x, y \rangle_\varepsilon^p dx dy \right)^{n}
	\end{equation}
for any nonzero nonnegative continuous function $g:\R^n \to \R$ with compact support.
The best constant $\constantd$ is given by
\[\constantd = \constanta \constantb n^{-p}.\]
Moreover, Inequality \eqref{ineq_Renyi_asim} is sharp and equality holds for some $g$ and $\c$ if, and only if, $f(x) = a p_\l(|B(x-\c)|)$ and $g(x) = a' p_\l(|B^{-T}x|)$, where $a,a' \in \R$ and $B \in \gl_n$.
\end{theorem}

The paper is organized as follows: In Section 2 we introduce the main tools and definitions. Section 3 is devoted to the existence of the Santal\'o point of $f$ (Theorem \ref{thm_existeSantaloLp}) and in Section 4 we prove Theorems \ref{thm_BSlp_func_asim} and \ref{thm_Renyi_asim}.

\section{Background in Convex Geometry}

This section is devoted to basic definitions and notations in Convex Geometry. For a comprehensive reference we refer to the book \cite{Sch}.

We recall that a convex body $K \subset \R^n$ is a convex compact subset of $\R^n$ with non-empty interior.

The support function $h_K$ is defined as

$$
h_K(y)=\max\{\langle y, z\rangle :\ z\in K\}\, .
$$
It describes the (signed) distance of supporting hyperplanes of $K$ to the origin and uniquely characterizes $K$. We also have the gauge $g$ and radial $r$ functions of $K$ defined respectively as

\[
g(K,y):=\inf\{\lambda>0 :\  y\in \lambda K\}\, ,\quad y\in\R^n\setminus\{0\}\, ,
\]

\[
r(K,y):=\max\{\lambda>0 :\ \lambda y\in K\}\, ,\quad y\in\R^n\setminus\{0\}\, .
\]
Clearly, $g(K,y) = r(K,y)^{-1}$. We also recall that $g(K,\cdot)$ is actually a norm when the convex body $K$ is centrally symmetric, i.e. $K=-K$, and the unit ball with respect to $g(K,\cdot)$ is just $K$.
On the other hand, a general norm on $\R^n$ is uniquely determined by its unit ball, which is a centrally symmetric convex body.

For a convex body $K\subset \R^n$ we define the polar body, denoted by $K^\circ$, by

\[
K^\circ:=\{y\in\R^n :\ \langle y,z \rangle\leq 1\quad \forall z\in K\}\, .
\]
Evidently, $h_K^{-1}(\cdot) = r(K^{\circ}, \cdot)$. It is also easy to see that $(\lambda K)^\circ=\frac{1}{\lambda}K^\circ$ for all $\lambda>0$. A simple computation using polar coordinates shows that

\begin{equation}
	\label{not_polares}
	\vol(K)=\frac{1}{n}\int_{\sp^{n-1}}r^n(K,\xi)d\xi=\frac{1}{n}\int_{\sp^{n-1}} g(K,\xi)^{-n} d\xi\, .
\end{equation}

For a given convex body $K\subset\R^n$ we find in the literature many bodies associated to it. In particular, Lutwak and Zhang introduced in \cite{LZ} for a convex body $K$ its $L_p$ centroid body $\Gamma_pK$. This body is defined by

\[
h_{\Gamma_pK}^p(y):=\frac{1}{c_{n,p}\vol(K)}\int_{K}|\langle y,z\rangle|^p dz\quad \mbox{ for }y\in\R^n\, ,
\]
where

\[
c_{n,p} = \frac{\omega_{n+p}}{\omega_2 \omega_n \omega_{p-1}}\, .
\]

There are some other normalizations of the $L_p$ centroid body in the literature and the previous one is made so that $\Gamma_p \B^n_2 =\B^n_2$ for the Euclidean unit ball $\B^n_2$ in $\R^n$ centered at the origin.

The definition of $\Gamma_pK$ can also be written as
\[
h_{\Gamma_pK}^p(y)=\frac{1}{n c_{n-2,p}\vol(K)}\int_{\sp^{n-1}} r(K,\xi)^{n+p}|\langle y,\xi \rangle|^p d\xi\quad \mbox{ for }y\in\R^n\, .
\]

For $\varepsilon \in [0,1]$ it is also convenient to define the asymmetric moment body of $K$ as
\begin{align*}
	h(\Me K, y)^p
	&= \int_K \langle y, z \rangle_\varepsilon^p dz\\
	&= \frac 1{n+p} \int_{\sp^{n-1}} r(K,\xi)^{n+p} \langle y, \xi \rangle_\varepsilon^p d \xi \quad \mbox{ for }y\in\R^n\, .
\end{align*}

The family of asymmetric centroid inequalities \eqref{ineq_BSlp_asim} was proved by Haberl and Schuster in \cite{HS}.
Notice that $h(\Me K, \cdot)$ as given in the first formula is always a convex function regardless of $K \subset \R^n$ being convex or not.
\begin{definition}
	\label{def_centroidf}
	For a positive continuous function $f:\R^n \to \R$ with compact support we define the convex set $\Me f$ by the support function
	\[
      h(\Me f, y) = \left(\int_{\R^n} f(z) \langle y, z \rangle_\varepsilon^p dz \right)^{\frac 1p}.
    \]
\end{definition}

We recall that $\bold{c}(K)$ represents the center of mass of the set $K$. For any function $f:\R^n \to \R$ and $s \in \R^n$, we also denote $f^s(x) = f(x - s)$. In order to prove Theorems \ref{thm_BSlp_func_asim} and \ref{thm_Renyi_asim} we shall make use of the following key result:
\begin{theorem}
	\label{thm_existeSantaloLp}
	Let $f:\R^n \to \R$ be a nonzero nonnegative continuous function with compact support and $\varepsilon \in \{0, \frac 12, 1\}$. Then, there exists a point $\c \in \R^n$ such that $0$ is the Santal\'o point of $\Me \Med (f^{\c})$. In other words, the point $\c$ satisfies $\bold{c}(\Med \Med (f^{\c})) = 0$.
\end{theorem}
Thus in the proofs of Theorems \ref{thm_BSlp_func_asim} and \ref{thm_Renyi_asim} it will be enough to translate $f$ by $\c$, although one could also minimize the right-hand side with respect to the translates of $f$ as in \eqref{ineq_BSlp_asim_new}.
For $\varepsilon=0$, $f = \chi_K$ and $p \to + \infty$ we have $\Med \Med (f^{s}) \to \co((K-s) \cup \{0\} )$ and Inequality \eqref{ineq_BSlp_asim_new} recovers \eqref{ineq_BSasim} with $s$ replaced by $\bold{c} = \bold{c}(K)$ (not by the Santal\'o point), see Inequality \eqref{ineq BScentr} of the introduction.
Therefore we must consider $\c$ to be an $L_p$-functional analogue of the center of mass.

\begin{remark}
	\label{rem_uniqueness}
	Although we are tempted to view $\c$ as an $L_p$-functional affine-invariant point, we do not know if this point is in fact unique.
	We also do not know about the existence of such a point for $\varepsilon \not\in \{0, \frac 12, 1\}$.
	Another point we leave open is whether there is a unique point minimizing the right-hand side of \eqref{ineq_BSlp_func_asim} or even of \eqref{ineq_BSlp_asim_new} (this point would work as an $L_p$-Santal\'o point of $f$ and $K$, respectively). Finally, it would be interesting to undertand the connection between $\c$ and the functional Santal\'o point introduced in \cite{A-K-M}. These problems seem to be challenging.
\end{remark}

\section{Existence of the point $\c$}
In this section we prove Theorem \ref{thm_existeSantaloLp} which is divided in two steps.
Namely, we prove that the function $\mu$ defined in Lemma \ref{lemma_mu_cont} below is continuous, and later that it has at least one zero.

Let us consider the notations $\M f := M_{0,p} f$, $M_{-,p} f := M_{1,p} f$, $\M K := M_{0,p} K$ and $\langle \cdot, \cdot \rangle_+ := \langle \cdot, \cdot \rangle_0$.
Lemma \ref{lemma_center} motivates the following definition:
\begin{lemma}
	\label{lemma_mu_cont}
Let $f: \R^n \to \R$ be a nonzero nonnegative continuous function with compact support. The function $\mu: \R^n \to \R^n$ given by
\[ \mu(\alpha) = \int_{\Md\Md f^\alpha} z dz\, , \]
is continuous.
\end{lemma}

\begin{proof}
For $R > 0$, using the radial function $r$ of the set $\Md f^\alpha \cap \B_R$, consider the function $\gamma_R : \R^n \times \R^n \rightarrow \R$ given by

\[
\gamma_R(\alpha, y) = r( \Md f^\alpha \cap \B_R, y ) = \left( \max\{g(\Md f^\alpha, y), R^{-1}\} \right)^{-1}\, ,
\]
where the gauge function $g$ of $\Md f^\alpha$ is defined by

\[
g(\Md f^\alpha, y) := \left( \int_{\R^n} f^\alpha(z) \langle y, z \rangle_+^p dz \right)^{\frac 1p}\, .
\]
Here $\B_R:= \B^n_2(R,0)$ denotes the Euclidean ball of $\R^n$ of radius $R$ centered at the origin.

Notice that $g(\Md f^\alpha, y)$ is continuous on $(\alpha,y)$, so $\gamma_R$ is continuous and positive.

Introduce now the function $\delta_R : \R^n \times \R^n \rightarrow \R$ by

	\begin{align*}
		\delta_R(\alpha, y)
		&= g(\Md(\Md f^\alpha \cap \B_R), y)\\
		&= \left(\int_{\B_R \cap \Md f^\alpha} \langle y, z \rangle_+^p dz \right)^{\frac 1p}\\
		&= \left(\int_{\sp^{n-1}} \gamma_R(\alpha, \xi)^{n+p} \langle y, \xi \rangle_+^p d\xi \right)^{\frac 1p} .\\
	\end{align*}
	Since $\gamma_R(\alpha, \xi)$ varies continuously with respect to $\alpha$ (uniformly on $\xi \in \sp^{n-1}$), we have that $\delta_R$ is continuous.
Moreover, the sequence $\delta_R$ is positive and monotone increasing with respect to $R$.

We shall prove that for every number $T>0$ there is a constant $\varepsilon(T)>0$ such that $\delta_R(\alpha, \xi) \geq \varepsilon(T)$ for every $\alpha \in \B_T$, $\xi \in \sp^{n-1}$ and $R > 1$. To see this, observe that triangular inequality produces
\begin{align*}
h(\M f^\alpha, y) &= \left( \int_{\R^n} f^\alpha(z) \langle y, z \rangle_+^p dz \right)^{\frac 1p} \\
&\leq \|f\|_\infty \max\{\|z\|:\, z \in \supp f^\alpha\} \leq A + B T
\end{align*}
for some positive constants $A$ and $B$ independent of $\alpha$. Here we used that $\supp f^\alpha = \alpha + \supp f$ and $\alpha \in \B_T$.

Notice also that
\[
r(\Md f^\alpha, y) = h(\M f^\alpha, y)^{-1}\, ,
\]
so we have $\B_{ (A + B T)^{-1}} \subseteq \Md f^\alpha$ and

	\begin{equation}
		\label{lemma_mu_cont_bound_delta} \delta_R(\alpha, \xi) \geq \left( \int_{\B_{(A + B T)^{-1}}} \langle \xi, z \rangle_+^p dz \right)^{\frac 1p} =: \varepsilon(T)\, .
	\end{equation}
Remark that the definition of $\varepsilon(T)$ does not depend on $\xi \in \sp^{n-1}$ since the integral of the right-hand side of \eqref{lemma_mu_cont_bound_delta} is invariant by orthogonal transformations.

Let $\mu_R: \R^n \rightarrow \R$ be the function defined by

\[
\mu_R(\alpha) = \int_{\Md(\Md f^\alpha \cap \B_R)} z dz\, .
\]
Using polar coordinates, we get

\begin{align*}
\mu_{R}(\alpha)
&= \int_{\sp^{n-1}} \delta_R(\alpha, \xi)^{-n-1} \xi d \xi
\end{align*}
and so it follows that $\mu_R$ is continuous. In addition, its definition provides readily that $\mu_R \to \mu$ pointwise in $\R^n$ as $R \to +\infty$, and we claim the convergence is uniform in compact sets of $\R^n$, concluding the result. Assuming otherwise, there exist $\alpha_k \to \alpha_0$ and $R_k \to +\infty$ such that $| \mu_{R_k}(\alpha_k) - \mu(\alpha_0) | \geq \varepsilon$ for all $k$.
Since the sequence $(\alpha_k)$ is bounded, in view of \eqref{lemma_mu_cont_bound_delta} we may apply dominated convergence to

\begin{align*}
\mu_{R_k}(\alpha_k)
&= \int_{\sp^{n-1}} \delta_R(\alpha_k, \xi)^{-n-1} \xi d \xi
\end{align*}
and conclude that $\mu_{R_k}(\alpha_k) \to \mu(\alpha_0)$ which is absurd.
\end{proof}

We finish this section by proving Theorem \ref{thm_existeSantaloLp}. For $\varepsilon = \frac 12$ we take the point $\s(f)$ at the origin, since all centroid bodies are symmetric.
Besides, since $M_{-,p} f = -\M f$, we only need to consider the case $\varepsilon = 0$.

\begin{proof}[Proof of Theorem \ref{thm_existeSantaloLp}]
Let $a,b \in (0,1)$ be such that $a^2+b^2 > 1$. Take $D = \max\{ \|z\|:\, z \in \supp f\}$ and fix $T > 0$ large enough so that $T^2 - T D > a T (T+D)$.
We shall prove that

\[
\deg(\mu, \B_T, 0) = 1\, ,
\]
where $\deg$ is the Brouwer topological degree. This proves the lemma by the ``solution'' property of $\deg$.
We shall denote the unit vector $\bar v = \frac v {\|v\|}$ in the direction of $v \in \R^n \setminus \{0\}$.
Fix $\alpha \in \partial \B_T$.
For each $x \in \supp f^\alpha$, we have
\begin{align*}
\langle x, \alpha \rangle
&= T^2 + \langle x-\alpha, \alpha \rangle \\
&\geq T^2 - T D \\
&> a T (T+D)\\
&\geq a \|x\| \|\alpha\|\, ,
\end{align*}
so that $\langle \bar x, \bar \alpha \rangle > a$. Denote $C_t$ the cone defined by

\[
C_t(\alpha) = \{v \in \R^n:\ \langle \bar v, \bar \alpha \rangle > t\}\, .
\]
Choose any $x \in \supp f^{\alpha}$ and $y \in C_{b}(-\alpha)$ and let $z = \bar x - \bar \alpha \langle \bar x, \bar \alpha \rangle$ and $w = \bar y - \bar \alpha \langle \bar y, \bar \alpha \rangle$. Since $z \bot \alpha$ and $w \bot \alpha$, we have
\begin{align*}
\langle \bar x, \bar y \rangle
&= \langle z + \bar \alpha \langle \bar x, \bar \alpha \rangle, w + \bar \alpha \langle \bar y, \bar \alpha \rangle \rangle\\
&\leq \langle \bar x, \bar \alpha \rangle \langle \bar y, \bar \alpha \rangle + \|z\| \|w\|\\
&= \langle \bar x, \bar \alpha \rangle \langle \bar y, \bar \alpha \rangle + \sqrt{1 - \langle \bar x, \bar \alpha \rangle^2} \sqrt{1 - \langle \bar y, \bar \alpha \rangle^2}\\
&< \sqrt{1 - a^2} \sqrt{1 - b^2} - a b \\
&=\sqrt{a^2 b^2 + 1 - a^2 - b^2} - a b\\
& < 0\, ,
\end{align*}
so that $\langle x, y \rangle < 0$ whenever $x \in \supp f^{\alpha}$ and $y \in C_{b}(-\alpha)$. Therefore, for any such a point $y$, we have

\begin{align*}
h(\M f^\alpha, y) &= \left(\int_{\R^n} f^\alpha(z) \langle y, z \rangle_+^p dz \right)^{\frac 1p}\\
&= \left(\int_{\supp f^{\alpha}} f^\alpha(z) \langle y, z \rangle_+^p dz \right)^{\frac 1p} = 0\, .\\
\end{align*}
Then $C_b(-\alpha) \subseteq \Md f^\alpha$. Moreover, if $v \in \R^n$ satisfies $\langle v, \alpha \rangle < 0$, then $-\alpha \in C_b(-\alpha) \cap C_0(v)$ which is thus a non-empty open cone. Consequently,

\[
\int_{\Md f^\alpha} \langle v, z \rangle_+^p dz \geq \int_{C_b(-\alpha)} \langle v, z \rangle_+^p dz \geq \int_{C_b(-\alpha) \cap C_0(v)} \langle v, z \rangle_+^p dz = +\infty\, .
\]
This means that $r(\Md \Md f^\alpha, v) = 0$ for every such $v$ and thus $\Md \Md f^\alpha$ is contained in the closed half space $\overline{C_0(\alpha)}$.
Since $\Md \Md f^\alpha$ has nonempty interior (because it is a convex body) then $\mu(\alpha) \in C_0(\alpha)$.

The relation $\langle \mu(\alpha), \alpha \rangle > 0$ for every $\alpha \in \partial \B_T$ guarantees that the homotopy $\mu_t(\alpha) = (1-t) \mu(\alpha) + t \alpha$ is nonzero for every $t \in [0,1]$ and $\alpha \in \partial \B_T$, and the invariance of the topological degree implies $\deg(\mu, \B_T, 0) = 1$.

\end{proof}

\section{Proof of the main theorems}

\begin{proof}[Proof of Theorem \ref{thm_BSlp_func_asim}]
As mentioned in the introduction, since the expression
\[
	\left( \int_{\sp^{n-1}} \left(\int_{\R^n} f(x - s) \langle x, \xi \rangle_\varepsilon^p dx \right)^{-\frac np} d\xi \right)^{-p}
\]
is minimized when $\varepsilon = 0$ the second part of Theorem \ref{thm_BSlp_func_asim}, which is stated for $\varepsilon \in [0, \frac 12)$, reduces to that case. So, for the complete proof of theorem, it suffices to consider the cases $\varepsilon = 0$ and $\varepsilon = \frac 12$.

For $\varepsilon = 0$ we may invoke Theorem \ref{thm_existeSantaloLp} which guarantees the existence of a point $\s(f) \in \R^n$ such that $0$ is the Santal\'o point of $\Me \Med (f^{\s(f)})$. So, we simply replace $f$ by $f^{\s(f)}$ and assume $0 \in \R^n$ is the Santal\'o point of $\Me \Med f$. For $\varepsilon = \frac 12$ we have $\s(f) = 0$ since it corresponds to the symmetric case and no change is necessary.

Therefore it suffices to show that

\[
	\|f\|_{1}^{n+\l' p} \leq \constantb \|f\|_{\l}^{\l'p} \left( \int_{\sp^{n-1}} \left(\int_{\R^n} f(x) \langle x, \xi \rangle_\varepsilon^p dx \right)^{-\frac np} d\xi \right)^{-p}.
\]

Consider the set $\Med \Med f$ whose gauge function is
\begin{align*}
	g(\Med \Med f, y)
	&= h(\Me \Med f, y)\\
	&= \left( \frac 1{n+p} \int_{\sp^{n-1}} g(\Med f, \xi)^{-n-p} \langle y, \xi \rangle_\varepsilon^p d\xi \right)^{1/p}
\end{align*}
and compute
\begin{align*}
	\int_{\R^n} f(x) g(\Med \Med f, x)^p dx
	&= \frac 1{n+p}\int_{\R^n} f(x)  \int_{\sp^{n-1}} g(\Med f, \xi)^{-n-p} \langle x, \xi \rangle_\varepsilon^p d\xi  dx\\
	&= \frac 1{n+p} \int_{\sp^{n-1}} g(\Med f, \xi)^{-n-p} \int_{\R^n} f(x)  \langle x, \xi \rangle_\varepsilon^p dx d\xi\\
	&= \frac 1{n+p}\int_{\sp^{n-1}} g(\Med f, \xi)^{-n} d\xi\\
	&= \frac n{n+p} \vol(\Med f)\, .
\end{align*}

From Inequality \eqref{ineq_moment} applied to $f$ and $K = \Med \Med f$, we obtain
\begin{align*}
	\|f\|_{1}^{n+\lambda'p}
	&\leq \left( \frac n{n+p}\right)^{n} \constanta \|f\|_{\lambda}^{\lambda'p} \vol(\Med \Med f)^{p} \vol(\Med f)^{n}.
\end{align*}

Since $\Med f$ is a convex body and $0$ is the Santal\'o point of $\Me \Med f$ we may apply Inequality \eqref{ineq_BSlp_asim} with $K = \Med f$ to obtain
\begin{align*}
	\|f\|_{1}^{n+\lambda'p}
	&\leq \left( \frac n{n+p}\right)^{n} \constanta R_{n,p}^{p} \|f\|_{\lambda}^{\lambda'p} \vol(\Med f)^{-(n+p)} \vol(\Med f)^{n} \\
	&= \constantb \|f\|_{\lambda}^{\lambda'p} (n \vol(\Med f))^{-p}
\end{align*}
and conclude
\begin{align*}
	\|f\|_{1}^{n+\l'p}
	&\leq \constantb \|f\|_{\lambda}^{\lambda'p} \left( \int_{\sp^{n-1}} \left(\int_{\R^n} f(x) \langle x, \xi \rangle_\varepsilon^p dx \right)^{-\frac np} d\xi \right)^{-p}.
\end{align*}
\end{proof}

\begin{proof}[Proof of Theorem \ref{thm_Renyi_asim}]
As in the proof of Theorem \ref{thm_BSlp_func_asim} we may assume $\s(f)=0$. Consider the set $\Me f$ and compute
%Take the following function
%\[\tilde H_{f}(v) = h(\Me f, v)\]
%so
\[
\int_{\R^n} g(\Med f,x)^p g(x) dx = \int_{\R^n} \int_{\R^n} f(x) g(y) \langle x, y \rangle_\varepsilon^p dy dx\, .
\]
Then Inequality \eqref{ineq_moment} with $K = \Med f$ gives

\[
\|g\|_{1}^{n+\lambda'p} \leq  \constanta \|g\|_{\lambda}^{\lambda'p} \vol(\Med f)^{p} \left( \int_{\R^n} \int_{\R^n} f(x) g(y) \langle x, y \rangle_\varepsilon^p dy dx \right)^{n}\, .
\]
Combining with Theorem \ref{thm_BSlp_func_asim} for $f$,

\[
\|f\|_{1}^{n+\lambda'p} \leq \constantb \|f\|_{\lambda}^{{\lambda'p}}  (n \vol(\Med f))^{- p}
\]
we get

\begin{align*}
\|f\|_{1}^{n+\lambda'p} \|g\|_{1}^{n+\lambda'p}
	&\leq  d_{n,p,\l} \|f\|_{\lambda}^{{\lambda'p} } \|g\|_{\lambda}^{{\lambda'p}} \left( \int_{\R^n} \int_{\R^n} f(x) g(y) \langle x, y \rangle_\varepsilon^p dy dx \right)^{n}.\\
\end{align*}

\end{proof}

The equality cases of Theorems \ref{thm_BSlp_func_asim} and \ref{thm_Renyi_asim} follow directly from the equality cases of Inequalities \eqref{ineq_BSasim} and \eqref{ineq_moment}. This finishes the proof.\\

\noindent {\bf Acknowledgments:}
The first author was partially supported by Fapemig (APQ-01454-15). The second author was partially supported by CNPq (PQ 305650/2016-5) and the program Incentivo \`a produtividade em ensino e pesquisa from the PUC-Rio. The third author was partially supported by CNPq (PQ 306855/2016-0) and Fapemig (APQ 02574-16).


\begin{thebibliography}{99}

    \bibitem{A-K-M} Artstein-Avidan, S., Klartag, B., Milman, V. -\textit{On the Santalo point of a function and a functional form of the Santalo inequality},
Mathematika 51 (2004) no. 1-2, (2005) 33--48.

	\bibitem{BK} Bianchi, G., Kelly, M. - \textit{A Fourier analytic proof of the Blaschke-Santal\'o inequality}, Proc. Amer. Math. Soc. 143 (2015) 4901-4912.

    \bibitem{Blas} Blaschke, W. - \textit{Uber affine Geometrie 7: Neue Extremeigenschaften von Ellipse und Ellipsoid}, Wilhelm Blaschke Gesammelte Werke 3. Thales Verlag, Essen (1985).

    \bibitem{CNV} Cordero-Erausquin, D., Nazaret, B., Villani, C. - \textit{A mass-transportation approach to sharp Sobolev and Gagliardo-Nirenberg inequalities}, Adv. Math. 182 (2004) 307-332.

	\bibitem{HJM2} De N\'apoli, P. L., Haddad, J., Jim\'enez, C. H., Montenegro, M. - \textit{The sharp affine $L^2$ Sobolev trace inequality and variants}, Mathematische Annalen 370 (2017) 287-308.

	\bibitem{FM} Fradelizi, M., Meyer, M. - \textit{Some functional forms of Blaschke-Santal\'o inequality}, Math. Z. 256 (2007) 379-395.

	\bibitem{HS} Haberl C., Schuster F. - \textit{General $L_p$ affine isoperimetric inequalities}, J. Differential Geom. 83 (2009) 1-26.

	\bibitem{HJM1} Haddad, J., Jim\'enez, C. H., Montenegro, M. - \textit{Sharp affine Sobolev type inequalities via the $\L_p$ Busemann-Petty centroid inequality}, J. Funct. Anal. 271 (2016) 454-473.

	\bibitem{HJM3} Haddad, J., Jim\'enez, C. H., Montenegro, M. - \textit{Sharp affine weighted $\L_p$ Sobolev type inequalities}, arXiv:1708.09471.

    \bibitem{Hug} Hug, D. - \textit{Contributions to affine surface area}, Manuscripta Math. 91 (1996) 283-301.

	\bibitem{LiLe} Lin, Y., Leng, G. - \textit{On the functional Blaschke-Santal\'o inequality}, arXiv:1403.0299.

	\bibitem{Lud} Ludwig, M. - \textit{Minkowski valuations}, Trans. Amer. Mat. Soc. 357 (2004) 4191-4213.

	\bibitem{LYZ} Lutwak, E., Yang, D., Zhang, G. - \textit{$L_p$ affine isoperimetric inequalities}, J. Differential Geom. 56 (2000) 111-132.

	\bibitem{LYZp} Lutwak, E., Yang, D., Zhang, G. - \textit{Sharp affine $L_p$ Sobolev inequalities}, J. Differential Geom. 62 (2002) 17-38.

	\bibitem{RENYI} Lutwak, E., Yang, D., Zhang, G. - \textit{Moment-entropy inequalities}, Ann. Probab. 32 (2004) 757-774.
	
    \bibitem{LYZorl} Lutwak, E., Yang, D., Zhang, G. - \textit{Orlicz projection bodies}, Adv. Math. 223 (2010) 220-242.

	\bibitem{LZ} Lutwak, E., Zhang, G. - \textit{Blaschke-Santal\'o inequalities}, J. Differential Geom. 47 (1997) 1-16.

    \bibitem {Madi} Madiman, M., Melbourne, J., Xu, P. - \textit{Forward and reverse entropy power inequalities in Convex Geometry}, The IMA Volumes in Mathematics and its Applications 161 Springer, New York, 2017.
	
    \bibitem{MP} Meyer, M., Pajor, A. - \textit{On the Blaschke-Santal\'o inequality}, Arch. Math. (Basel) 55 (1990) 82-93.

    \bibitem{MR} Meyer, M., Reisner, S. - \textit{Shadow systems and volumes of polar convex bodies}, Mathematika 53 (2006) 129-148.
   \bibitem{M-W} Meyer, M., Werner, E. -\textit{Santal\'o regions of a convex body}, Trans. Amer. Math. Soc., 350 (11) (1998), pp. 4569-4591.

	\bibitem{Ng} Nguyen V. H. - \textit{New approach to the affine P\'olya-Szeg\H{o} principle and the stability version of the affine Sobolev inequality}, Adv. Math. 302 (2016) 1080-1110.

    \bibitem{Pet} Petty, C. M. - \textit{Affine isoperimetric problems}, Discrete geometry and convexity, New York Acad. Sci. (1985) 113-127.

    \bibitem{Prk} Pr\'ekopa A. - \textit{Logarithmic concave measures and functions}, Acta Sci. Math. 34 (1973) 334-343

    \bibitem{Santa} Santal\'o, L. A. - \textit{An affine invariant for convex bodies of $n$-dimensional space}, Port. Math. 8 (1949) 155-161.

	\bibitem{Sch} Schneider, R. - \textit{Convex Bodies: The Brunn-Minkowski Theory}, 2nd ed. Cambridge: Cambridge University Press, 2013.

	\bibitem{Wa} Wang, T. - \textit{On the Discrete Functional Lp Minkowski Problem}, Int. Math. Res. Not. 20 (2015) 10563-10585.

    \bibitem{Zh} Zhang, G. - \textit{The affine Sobolev inequality}, J. Differential Geom. 53 (1999) 183-202.
	
    \bibitem{Zh1} Zhu, G. - \textit{The Orlicz centroid inequality for star bodies}, Adv. Appl. Math. 48 (2012) 432-445.
\end{thebibliography}
\end{document}